\documentclass[a4paper,10pt]{scrartcl}

\usepackage{amsmath, amstext, paralist, amsthm, amssymb}

\allowdisplaybreaks

\newfont{\theoremfont}{cmssbx12 scaled 875}
\newtheoremstyle{Eins}{\topsep}{\topsep}{\itshape}{}{\theoremfont}{.}{5pt}{\thmname{#1}\thmnumber{ #2}\thmnote{ #3}}
\newtheoremstyle{Zwei}{\topsep}{\topsep}{}{}{\theoremfont}{.}{5pt}{\thmname{#1}\thmnumber{ #2}\thmnote{ #3}}

\newcommand{\myfootnote}[1]{
 \renewcommand{\thefootnote}{}
 \footnotetext{#1}
 \renewcommand{\thefootnote}{\arabic{footnote}}
 }

\theoremstyle{Eins}
\newtheorem{thm}{Theorem}[section]

\theoremstyle{Zwei}
\newtheorem{lem}[thm]{Lemma}
\newtheorem*{ack}{Acknoledgements}
\newtheorem{dfn}[thm]{Definition}
\newtheorem{rem}[thm]{Remark}

\begin{document}
\title{Notes on a proof of Bonet, Engli\v{s}\\and Taskinen}
\author{Sven-Ake Wegner}
\maketitle
\begin{abstract}\noindent{}We give a proof of the result \cite[Theorem 5]{BoEnTa2005} of Bonet, Engli\v{s}, Taskinen filling in several details and correcting some flaws.
\end{abstract}

\myfootnote{2010 Mathematical Subject Classification: primary 46E10, secondary 46A13, 46M40.}

\section{Preliminaries}

Let in the sequel $H(\mathbb{D})$ denote the space of all holomorphic functions on the open unit disc $\mathbb{D}$. A \textit{weight} is a strictly positive and continuous function on $\mathbb{D}$. For a weight $v$ we consider the space
$$
Hv_0(\mathbb{D}):=\big\{f\in H(\mathbb{D})\:;\:v|f| \text{ vanishes at } \infty \text{ on } \mathbb{D}\big\}
$$
which is a Banach space w.r.t.~the norm $\|f\|_{v}:=\sup_{z\in\mathbb{D}}v(z)|f(z)|$ for $f\in H(v)_0(\mathbb{D})$. In the sequel we will use the following well-known fact; for the sake of completeness we give a proof.

\begin{rem}\label{LOG-approx} Let $v$ be a radial weight which is decreasing on $[0,1[$. Assume that $(r_n)_{n\in\mathbb{N}}\subseteq[0,1[$ is a sequence with $r_n\nearrow1$ as $n\rightarrow\infty$. Let $g\in Hv_0(\mathbb{D})$ and put $g_n(z):=g(r_nz)$ for $z\in\mathbb{D}$. Then $g_n\rightarrow g$ holds in $Hv_0(\mathbb{D})$.
\end{rem}
\begin{proof}
We note first that $g_n\in Hv_0(\mathbb{D})$ holds. Moreover, $g_n\rightarrow g$ holds w.r.t.~the compact open topology: For $K\subseteq\mathbb{D}$ compact we select $0<R<1$ such that $K\subseteq \overline{B}_{R}(0)$ and estimate
\begin{align*}
\sup_{z\in K}|g(z)-g(r_nz)| & \leqslant \sup_{z\in K}\max_{\xi\in[r_nz,z]}|g'(\xi)||z-r_nz|\\
&\leqslant(1-r_n)\sup_{z\in K}\max_{\xi\in[r_nz,z]}|g'(\xi)|\\
&\leqslant(1-r_n)\sup_{z\in\overline{B}_{R}(0)}|g'(z)|\stackrel{n\rightarrow\infty}{\longrightarrow}0
\end{align*}
which yields the desired co-convergence.
\smallskip
\\Let now $\varepsilon>0$ be given. Since $g\in Hv_0(\mathbb{D})$ there exists $0<R_0<1$ such that $v(z)|g(z)|\leqslant\frac{\varepsilon}{3}$ for each $|z|\geqslant R_0$. We select $0<R_0<R_1<1$. Then in particular $\sup_{|z|\geqslant R_1}v(z)|g(z)|\leqslant\frac{\varepsilon}{3}$ holds. By the above we may select $N$ such that $\sup_{|z|\leqslant R_1}v(z)|g(z)-g(r_nz)|\leqslant\frac{\varepsilon}{3}$ holds for $n\geqslant N$. By increasing $N$ we may assume that $r_nR_1\geqslant R_0$ for $n\geqslant N$. Now we get
\begin{align*}
\sup_{z\in\mathbb{D}}v(z)|g(z)-g(r_nz)| & \leqslant \sup_{|z|\leqslant R_1}v(z)|g(z)-g(r_nz)|+\sup_{|z|\geqslant R_1}v(z)|g(z)-g(r_nz)|\\
&\leqslant {\textstyle\frac{\varepsilon}{3}}+\sup_{|z|\geqslant R_1}v(z)|g(z)|+\sup_{|z|\geqslant R_1}v(z)|g(r_nz)|\\
&\leqslant {\textstyle\frac{\varepsilon}{3}}+{\textstyle\frac{\varepsilon}{3}}+\sup_{|\xi|\geqslant r_nR_1}v(\xi)|g(\xi)|\\
&\leqslant {\textstyle\frac{\varepsilon}{3}}+{\textstyle\frac{\varepsilon}{3}}+\sup_{|\xi|\geqslant R_0}v(\xi)|g(\xi)|\\
&\leqslant {\textstyle\frac{\varepsilon}{3}}+{\textstyle\frac{\varepsilon}{3}}+{\textstyle\frac{\varepsilon}{3}}=\varepsilon
\end{align*}
for $n\geqslant N$.
\end{proof}

\section{The result of Bonet, Engli\v{s} and Taskinen}

Let from now on $\mathcal{V}=(v_k)_{k\in\mathbb{N}}$ be a decreasing sequence of strictly positive and continuous functions (weights) on the unit disc $\mathbb{D}$ of the complex plane. For every $n\in\mathbb{N}$ we put $r_{n}:=1-2^{-2^{n}}$, $r_0:=0$ and $I_n:=[r_{n},r_{n+1}]$.

\begin{dfn} We say that the sequence $\mathcal{V}=(v_{k})_{k\in\mathbb{N}}$ satisfies condition (LOG) if each weight in the sequence is radial and approaches monotonically $0$ as $r\nearrow1$ and there exist constants $0<a<1<A$ such that the conditions
\begin{compactitem}\vspace{5pt}
\item[(LOG 1)] $A\cdot v_{k}(r_{n+1})\geqslant v_{k}(r_{n})$ and\vspace{3pt}
\item[(LOG 2)] $v_{k}(r_{n+1})\leqslant a\cdot v_{k}(r_{n})$
\end{compactitem}\vspace{5pt}
are satisfied for all $n$ and $k\in\mathbb{N}$.
\end{dfn}

We consider the \textit{weighted LB-space of holomorphic functions} $\mathcal{V}_0H(\mathbb{D})\!=\!\operatorname{ind}_{k}H(v_k)_0(\mathbb{D})$ and its \textit{projective hull} $H\overline{V}_{\!0}(\mathbb{D})=\operatorname{proj}_{\overline{v}\in\overline{V}}H(\overline{v})_0(\mathbb{D})$, where
$$
\overline{V}:=\big\{\overline{v}\:;\:\overline{v} \text{ is a weight on }\mathbb{D}\text{ such that }\forall\:k\;\exists\:\alpha_k>0\colon \overline{v}\leqslant {\textstyle\inf_{k}}\alpha_kv_k\big\}.
$$
Projective hulls were introduced by Bierstedt, Meise, Summers in the seminal article \cite{BMS1982} and are the starting point for the so-called \textit{projective description problem}. For details we refer to the latter article and to the survey \cite{Bierstedt2001} of Bierstedt. At this point we only mention that in the terminology of \cite{Bierstedt2001,BMS1982} \textit{projective description holds} for the space $\mathcal{V}_0H(\mathbb{D})$, if $\mathcal{V}_0H(\mathbb{D})\subseteq H\overline{V}_{\!0}(\mathbb{D})$ is a topological subspace.
\smallskip
\\Let us now state the result of Bonet, Engli\v{s}, Taskinen, which states that projective description holds if we assume the sequence $\mathcal{V}$ to satisfy condition (LOG).

\begin{thm}\label{LOG-main-result}(\cite[Theorem 5]{BoEnTa2005}) If the weight system $\mathcal{V}$ satisfies condition (LOG), then $\mathcal{V}_0H(\mathbb{D})$ is a topological subspace of $H\overline{V}_{\!0}(\mathbb{D})$.
\end{thm}

\begin{proof} We put $D:=\max(\sum_{k\in\mathbb{N}}a^k,\,\sup_{n>t+2}2^{-n}A^{n-t}2^{-2^{n-1}})$, where $a$ and $A$ are the constants of (LOG 1) and (LOG 2). Now put $C:=2A^2(D+A^2)+ 4(A^2+2D)$.
\smallskip
\\For every $k\in\mathbb{N}$ we put
$$
U_k:=\{f\in H(v_k)_0(\mathbb{D})\:;\:\|f\|_{v_k}\leqslant1\}.
$$
Let a 0-neighborhood $B=\Gamma(\cup_{k\in\mathbb{N}}b_kU_k)$ in $\mathcal{V}_0(\mathbb{D})$ be given, where $\Gamma$ stands for the absolutely convex hull and $b_k$ is a positive constant for every $k$. Let us define the decreasing weight
$$
\overline{v}:=\inf_{k\in\mathbb{N}}a_k^{-1}v_k(z),
$$
with $a_k<2^{-(k+2)}C^{-1}b_k$. Then $\overline{v}$ is radial, continuous, non-increasing and the infimum is a minimum on compact subsets of $\mathbb{D}$, see Bierstedt, Meise, Summers \cite{BMS1982}. Assuming that $f\in\mathcal{V}_0H(\mathbb{D})$ satisfies $\|f\|_{\overline{v}}\leqslant1$, we show that $f\in\overline{B}$, where the closure is taken in $\mathcal{V}_0H(\mathbb{D})$. We define $f_{r_{n}}$ by $f_{r_{n}}(z)=f(r_{n}z)$. By \ref{LOG-approx} it is enough to show that $f_{r_n}\in B$ for sufficiently large $n\in\mathbb{N}$.
\smallskip
\\Since $\overline{v}$ is nonincreasing we get
\begin{align}
\inf_{|z|\in I_{n}}\overline{v}(z)=\overline{v}(r_{n+1})\geqslant \overline{v}(r_{n+2})=\inf_{|z|\in I_{n+1}}\overline{v}(z)\stackrel{\text{\tiny (LOG 1)}}{\geqslant} A^{-2}\overline{v}(r_{n}).
\end{align}
For every $n$ we pick $k(n)\in\mathbb{N}$ such that
\begin{align}
\overline{v}(r_{n})=a_{k(n)}^{-1}v_{k(n)}(r_{n})=a_{k(n)}^{-1}\sup_{|z|\in I_{n}}v_{k(n)}(z).
\end{align}
We fix $\nu\in\mathbb{N}$ and for every $m\in\mathbb{N}$ we define
$$
N_{m}:=\{n\in\mathbb{N}\:;\:n\leqslant \nu\text{ and } k(n)=m\}.
$$
Hence, the set $\{n\in\mathbb{N}\:;\:n\leqslant\nu\}$ is a disjoint union of the sets $N_m$; some of them may be empty sets. For each $n\geqslant1$ we put $g_{n}(z):=f(r_{n+1}z)-f(r_{n}z)$ and $g_0(z):=f(0)$. For $m\in\mathbb{N}$ we define
$$
h_{m}:=\sum_{n\in N_{m}}g_{n}
$$
if $N_m\not=\emptyset$ and $h_{m}=0$ otherwise. We have
\begin{align*}
(g_0+\sum_{m\in\mathbb{N}}h_m)(z) & = g_0(z)+\sum_{m\in\mathbb{N}}\sum_{n\in N_m}(f(r_{n+1}z)-f(r_nz))\\
& = f(0)+\sum_{n=0}^{\nu}(f(r_{n+1}z)-f(r_nz)) \\
& = f(0)+\sum_{n=0}^{\nu}f(r_{n+1}z)-\sum_{n=0}^{\nu}f(r_nz)\\
& = f(0)+\sum_{n=1}^{\nu}f(r_nz)+f(r_{\nu+1}z)-\sum_{n=1}^{\nu}f(r_nz)-f(0)\\
&=f(r_{\nu+1}z).
\end{align*}
The constant function $g_0$ belongs to $H(v_{k(0)})_0(\mathbb{D})$ and
$$
|g_0(z)|=|f(0)|\leqslant a_{k(0)}v_{k(0)}(0)^{-1}
$$
holds by (2), hence $g_0\in 2^{-2}2^{-k(0)}b_{k(0)}U_{k(0)}$. The main part of the proof is to show the following lemma.

\begin{lem}\label{lemma}(\cite[Lemma 2]{BoEnTa2005}) In the notation above $h_m\in 2^{-(m+2)}b_mU_m$ holds for all $m\in\mathbb{N}$.
\end{lem}

The lemma finishes the proof of \ref{LOG-main-result}: Since $h_m\in 2^{-(m+2)}b_mU_m$, we have
\begin{align*}
f_{r_{\nu}}=g_0+\sum_{m\in\mathbb{N}}h_m&\in 2^{-2}2^{-k(0)}b_{k(0)}U_{k(0)}+ \sum_{m\in\mathbb{N}}2^{-m}b_mU_m\\
&\subseteq 2^{-1}\sum_{m\in\mathbb{N}}2^{-m}b_mU_m\\
&\subseteq \Gamma(\cup_{m\in\mathbb{N}}b_mU_m)\\
&=B.\qedhere
\end{align*}
\end{proof}
\medskip
\begin{proof}\textit{(of Lemma 2.3)} We fix $m\in\mathbb{N}$, pick $n\in N_m$ and estimate $|g_n(z)|$ for different $z$.
\begin{compactitem}
\item[1.] Assume first $|z|\geqslant r_{n-1}$ (where we put $r_{n-1}:=r_0$ for $n=0$).\vspace{3pt}
\begin{compactitem}
\item[a.] Let $n\geqslant2$. Then we have
\begin{align*}
|r_{n}z|=|r_{n}||z|\geqslant|r_{n}||r_{n-1}| & = (1-2^{-2^{n}})(1-2^{-2^{n-1}})\\
&= 1-2^{-2^{n-1}}-2^{-2^{n}}+2^{-2^{n}}\cdot2^{-2^{n-1}}\\
&\geqslant 1-2^{-2^{n-1}}-2^{-2^{n}}\\
&\geqslant 1-2^{-2^{n-1}}-2^{-2^{n-1}}\\
&= 1-2\cdot 2^{-2^{n-1}}\\
&\geqslant 1-2^{-2^{n-2}}\\
&=r_{n-2}.
\end{align*}
Since $r_{n}\leqslant r_{n+1}$ and $|z|\leqslant1$ we get
$$
r_{n-2}\leqslant|r_{n}z|\leqslant |r_{n+1}z|\leqslant r_{n+1}\;\;\text{ for } n\geqslant2.
$$
Since $\|f\|_{\overline{v}}\leqslant1$, we have $|f(z)|\leqslant \overline{v}(z)^{-1}$ on $\mathbb{D}$. Thus we get by the above, since $\overline{v}$ is non-increasing and by (1)
\begin{align*}
|g_{n}(z)|&\stackrel{\text{\tiny dfn}}{=}|f(r_{n+1})-f(r_{n})|\\
 & \leqslant |f(r_{n}z)|+f(r_{n+1}z)|\\
& \leqslant \overline{v}(r_{n}z)^{-1}+\overline{v}(r_{n+1}z)^{-1}\\
& \leqslant \;\;\;\;2\!\!\!\!\!\!\!\sup_{r_{n-2}\leqslant r\leqslant r_{n+1}}\!\!\!\!\!\!\!\overline{v}(r)^{-1}\\
& = \;\;\;\:2\!\!\!\!\!\!\!\!\!\!\sup_{r\in I_{n-2}\cup I_{n-1}\cup I_{n}}\!\!\!\!\!\!\!\!\!\!\!\overline{v}(r)^{-1}\\
& =2\max\big(\sup_{r\in I_{n-2}}\overline{v}(r)^{-1},\,\sup_{r\in I_{n-1}}\overline{v}(r)^{-1},\,\sup_{r\in I_{n}}\overline{v}(r)^{-1}\big)\\
&\leqslant 2 \overline{v}(r_{n+1})^{-1}\\
&\leqslant 2 A^2 \overline{v}(r_{n})^{-1}\\
&= 2A^2 a_{m}v_{m}(r_{n})^{-1}
\end{align*}
where the last equality follows since $\overline{v}(r_{n})=a_{k(n)}^{-1}v_{k(n)}(r_{n})$ and $n\in N_{m}$ implies $i(n)=m$ (cf.~(2)).\vspace{3pt}
\item[b.] Let $n=1$. In this case we have
\begin{align*}
|g_1(z)|=|f(r_2z)-f(r_1z)|&\leqslant |f(r_2z)|+|f(r_1z)|\\
&\leqslant \overline{v}(r_2z)^{-1}+\overline{v}(r_1z)^{-1}\\
&\leqslant 2\sup_{r_0\leqslant r\leqslant r_2}\overline{v}(r)^{-1}\\
&=2\sup_{r\in I_0\cup I_1}u(r)^{-1}\\
&=2\max\big(\sup_{r\in I_0}u(r)^{-1},\,\sup_{r\in I_1}\overline{v}(r)^{-1}\big)\\
&=2\overline{v}(r_2)^{-1}\\
&\stackrel{\scriptscriptstyle(1)}{\leqslant}2 A^2 \overline{v}(r_1)^{-1}\\
&=2A^2a_{m}v_{m}(r_1)^{-1}
\end{align*}
where the last equality follows as above.\vspace{3pt}
\item[c.] Let $n=0$. Then we have $|g_{n}(z)|=|f(0)|$ and $\|f\|_{\overline{v}}\leqslant1$ which implies in particular $\overline{v}(0)|f(0)|\leqslant1$, i.e.
\begin{align*}
|g_{n}(z)|=|f(0)|  \leqslant \overline{v}(0)^{-1}&= \overline{v}(r_0)^{-1}\\
& = a_{k(0)}v_{k(0)}(r_0)^{-1}\\
& \leqslant 2A^2a_{k(n)}v_{k(n)}(r_{n})^{-1}\\
& =2A^2a_{m}v_{m}(r_{n})^{-1}
\end{align*}
 by (2), since $A>1$ and by our selection $n\in N_{m}$.\vspace{3pt}
\end{compactitem}
To sum up the results of the cases a., b.~and c., we have
$$
|g_{n}(z)|\leqslant 2A^2 a_{m}v_{m}(r_{n})^{-1}
$$
for $|z|\geqslant r_{n-1}$ and $n\geqslant0$.\vspace{5pt}
\item[2.] Assume now that $n>t+1$ and $|z|\in I_t$, i.e.~$r_t\leqslant|z|\leqslant r_{t+1}$. Then we have $|g_{n}(z)|=|f(r_{n}z)-f(r_{n+1}z)|$ by definition. By the mean value theorem there exists $\xi$ between $r_{n}z$ and $r_{n+1}z$ with
$$
|f(r_{n}z)-f(r_{n+1})|=|f'(\xi)||r_{n}z-r_{n+1}z|\leqslant|f'(\xi)||r_{n}-r_{n+1}|.
$$
Hence we may estimate
\begin{align*}
|g_{n}(z)| & \leqslant \sup_{|r_{n}z|\leqslant|\xi|\leqslant|r_{n+1}z|}|f'(\xi)||r_{n}-r_{n+1}|\\
&\leqslant \sup_{r_{n}r_t\leqslant|\xi|\leqslant r_{n+1}r_{t+1}}|f'(\xi)|2^{-2^{n}},
\end{align*}
since $|r_{n+1}-r_{n}|=1-2^{-2^{n+1}}-1+2^{-2^{n}}\leqslant 2^{-2^{n}}$. $n>t+1$, i.e.~$t<n-1$ implies $|\xi|\leqslant r_{n+1}r_{t+1}<r_{t+1}\leqslant r_{n}$ and we thus may use the Cauchy formula
\begin{align}
|f'(\xi)|\leqslant{\textstyle\frac{1}{2\pi}}\int_{|\eta|=r_{n}}{\textstyle\frac{|f(\eta)|}{|\eta-\xi|^2}}d\eta
\end{align}
to estimate $|f'(\xi)|$. We have $|f(\eta)|\leqslant \overline{v}(\eta)^{-1}=\overline{v}(r_{n})^{-1}$, since $\|f\|_{\overline{v}}\leqslant1$ and $\overline{v}$ is radial. Now we estimate $\frac{1}{|\eta-\xi|^{2}}$.\vspace{3pt}
\begin{compactitem}
\item[a.] Let $n>t+2$. That is, $n\geqslant t+3$, i.e.~$t\leqslant n-3$. Hence $|\xi|\leqslant r_{n+1}r_{t+1}\leqslant r_{n+1}r_{n-2}\leqslant r_{n-2}$. Now, $|\eta-\xi|\geqslant\big||\eta|-|\xi|\big|\geqslant |\eta|-|\xi|\geqslant r_{n}-r_{n-2}=1-2^{-2^{n}}-1+2^{-2^{n-2}}=2^{-2^{n-2}}-2^{-2^{n}}$. We claim that $2^{-2^{n-2}}-2^{-2^{n}}\geqslant 2^{-1}2^{-2^{n-2}}$ holds. We clearly have $2^{n}-2^{n-2}\geqslant1$, i.e.~$2^{n}-1\geqslant 2^{n-2}$ and thus $2^{2^{n}-1}\geqslant 2^{2^{n-2}}$, therefore $2^{1-2^{n}}\leqslant2^{-2^{n-2}}$ and thus $-2\cdot2^{2^{n}}=-2^{1-2^{n}}\geqslant-2^{-2^{n-2}}$. This implies $2\cdot2^{-2^{n-2}}-2\cdot2^{-2^{n}}\geqslant2\cdot2^{-2^{n-2}}-2^{-2^{n-2}}=2^{-2^{n-2}}$ which shows the claim. Thus we have $|\eta-\xi|\geqslant 2^{-1}2^{-2^{n-2}}$ hence $\frac{1}{|\eta-\xi|}\leqslant 2\cdot2^{2^{n-2}}$ which yields $\frac{1}{|\eta-\xi|^2}\leqslant 2^2\cdot2^{2\cdot2^{n-2}}=4\cdot2^{2^{n-1}}$. Now we get
$$
|f'(\xi)|\leqslant{\textstyle\frac{2\pi r_{n}}{2\pi}}\cdot4\cdot2^{2^{n-1}}\overline{v}(r_{n})^{-1}\leqslant 4\cdot2^{2^{n-1}}\overline{v}(r_{n})^{-1}
$$
from (3) since $r_{n}\leqslant1$ and can continue the estimate of $|g_{n}(z)|$, i.e.
\begin{align*}
|g_{n}(z)|& \leqslant 4\cdot 2^{2^{n-1}}2^{-2^{n}}\overline{v}(r_{n})^{-1}\\
               & = 4\cdot 2^{2^{n-1}-2^{n}}\overline{v}(r_{n})^{-1}\\
               & = 4\cdot 2^{2^{n-1}(1-2^{1})}\overline{v}(r_{n})^{-1}\\
               & = 4\cdot2^{-2^{n-1}}\overline{v}(r_{n})^{-1}\\
               & = 4\cdot2^{2^{n-1}}a_{m}v_{m}(r_{n})^{-1}
\end{align*}
where the last equality is obtained as in the previous cases.\vspace{3pt}
\item[b.] Let $n=t+2$, that is $t=n-2$ and hence $|\xi|\leqslant r_{n+1}r_{t+1}\leqslant r_{n+1}r_{n-1}\leqslant r_{n-1}$.  Similar to the above we have $|\eta-\xi|\geqslant r_{n}-r_{n-1}=1-2^{-2^{n}}-1+2^{-2^{n-1}}=2^{-2^{n-1}}-2^{-2^{n}}$ and we claim that $2^{-2^{n-1}}-2^{-2^{n}}\geqslant 2^{-1}2^{-2^{n-1}}$ holds. We clearly have $2^{n}-2^{n-1}\geqslant1$, i.e.~$2^{n}-1\geqslant 2^{n-1}$ and thus $2^{2^{n}-1}\geqslant 2^{2^{n-1}}$, therefore $2^{1-2^{n}}\leqslant2^{-2^{n-1}}$ and thus $-2\cdot2^{2^{n}}=-2^{1-2^{n}}\geqslant-2^{-2^{n-1}}$. This implies $2\cdot2^{-2^{n-1}}-2\cdot2^{-2^{n}}\geqslant2\cdot2^{-2^{n-1}}-2^{-2^{n-1}}=2^{-2^{n-1}}$ which shows the claim. Similar to the above, we get $|\eta-\xi|\geqslant2^{-1}2^{-2^{n-1}}$ and hence $\frac{1}{|\eta-\xi|}\leqslant2\cdot2^{2^{n-1}}$ which yields $\frac{1}{|\eta-\xi|^2}\leqslant 2^2\cdot 2^{2\cdot2^{n-1}}=4\cdot2^{2^{n}}$. We get
$$
|f'(\xi)|\leqslant{\textstyle\frac{2\pi r_{n}}{2\pi}}\cdot4\cdot2^{2^{n}}\overline{v}(r_{n})^{-1}\leqslant 4\cdot2^{2^{n}}\overline{v}(r_{n})^{-1}
$$
from (3) since $r_{n}\leqslant1$ and can also in this case continue the estimate of $|g_{n}(z)|$, i.e.
$$
|g_{n}(z)|\leqslant 4\cdot 2^{2^{n}}\overline{v}(r_{n})^{-1}2^{-2^{n}}=4\overline{v}(r_{n})^{-1}=4a_{m}v_{m}(r_{n})^{-1}
$$
by the choice $n\in N_{m}$.\vspace{3pt}
\end{compactitem}
Now we use (LOG 1) $(n-t)-$times to obtain
$$
v_{m}(r_t)\leqslant Av_{m}(r_{n+1})\leqslant\cdots\leqslant A^{n-t}v_{m}(r_{t+n-t})= A^{n-t}v_{m}(r_{n}).
$$
Since $|z|\geqslant r_t$ and because $v_{m}$ is radial and decreasing for $r\nearrow 1$ we have $v_{m}(r_t)\geqslant v_{m}(z)$ and thus we get $v_{m}(z)\leqslant v_{m}(r_{t})\leqslant A^{n-t}v_{m}(r_{n})$, which finally yields $v_{m}(r_{n})^{-1}\leqslant A^{n-t}v_{m}(z)^{-1}$. We continue the estimates in a.~and b.
\smallskip
\begin{compactitem}
\item[c.] Let $n>t+2$. From the latter and our estimate in a.~we get $|g_{n}(z)|\leqslant 4a_{m}v_{m}(z)^{-1}A^{n-t}2^{-2^{n-1}}$. By our selection of $D$ we get $A^{n-t}2^{-2^{n-1}}\leqslant D2^{-n}$ and therefore $|g_{n}(z)|\leqslant4\cdot2^{-n}Da_{m}v_{m}(z)^{-1}$.\vspace{3pt}
\item[d.] Let $n=t+2$. Then the above yields $|g_{n}(z)|\leqslant 4a_{m}v_{m}(z)^{-1}A^{2}$.
\end{compactitem}
\smallskip
To sum up the results of 2., we have
$$
|g_{n}(z)|\leqslant 4a_{m}v_{m}(z)^{-1}\,
\begin{cases}
\;2^{-n}D & \text{if }\;n>t+2\\
\;\;\,A^{2} & \text{if }\;n=t+2
\end{cases}
$$
for $|z|\in I_t$ and $n$ as indicated above.
\end{compactitem}
\medskip
To complete the proof, let now $z\in\mathbb{D}$ be arbitrary. We select $t\in\mathbb{N}$ such that $|z|\in I_t=[r_t,r_{t+1}]$. Then
$$
|h_{m}(z)|\stackrel{\text{\tiny dfn}}{=}\big|\sum_{n\in N_{m}}g_{n}(z)\big|\leqslant\sum_{\stackrel{n\in N_{m}}{\scriptscriptstyle n\leqslant t+1}}|g_{n}(z)|+\sum_{\stackrel{n\in N_{m}}{\scriptscriptstyle n>t+1}}|g_{n}(z)|\,=:\,G_{m}(z)+H_{m}(z).
$$
\begin{compactitem}
\item[(i)] We consider $G_{m}(z)$, that is all occuring $n$ satisfy $0\leqslant n\leqslant t+1$ and $n\in N_{m}$. Thus we have $n-1\leqslant t$, hence $|z|\geqslant r_t\geqslant r_{n-1}$ (remember that we defined $r_{-1}:=r_0=0$). By the estimate obtained in 1.~we therefore have
$$
G_{m}(z)\stackrel{\text{\tiny dfn}}{=}\sum_{\stackrel{n\in N_{m}}{\scriptscriptstyle n\leqslant t+1}}|g_{n}(z)|\leqslant \sum_{\stackrel{n\in N_{m}}{\scriptscriptstyle n\leqslant t+1}}2A^2a_{m}v_{m}(r_{n})^{-1}.
$$
(LOG 2) implies $v_{m}(r_{n+1})\leqslant a v_{m}(r_{n})$, i.e.~$v_{m}(r_{n})^{-1}\leqslant a v_{m}(r_{n+1})^{-1}$ for arbitrary $n$. Iterating this estimate $t-n$ times for a fixed $n\leqslant t$ we get
$$
u_{m}(r_{n})^{-1}\leqslant a u_{m}(r_{n+1})^{-1}\leqslant\cdots\leqslant a^{t-n}u_{m}(r_{n+t-n})^{-1}=a^{t-n}u_{m}(r_{t})^{-1}.
$$
With the latter we may estimate
\begin{align*}
\sum_{\stackrel{n\in N_{m}}{\scriptscriptstyle n\leqslant t+1}}2A^2a_{m}v_{m}(r_{n})^{-1}
& \leqslant \sum_{n\leqslant t+1}2A^2a_{m}v_{m}(r_{n})^{-1}\\
& = 2A^2a_{m}\big(\sum_{n=0}^tv_{m}(r_{n})^{-1} + v_{m}(r_{t+1})^{-1}\big) \\
&\leqslant 2A^2a_{m}\big(\sum_{n=0}^ta^{t-n}v_{m}(r_t)^{-1} + A^2u_{m}(r_{t})^{-1}\big) \\
&= 2A^2a_{m}v_{m}(r_t)^{-1}\big(\sum_{\sigma=0}^ta^{\sigma} + A^2\big) \\
&\leqslant 2A^2a_{m}v_{m}(r_t)^{-1}\big(\sum_{\sigma=0}^{\infty}a^{\sigma} + A^2\big) \\
&\leqslant 2A^2(D+A^2)a_{m}v_{m}(z)^{-1}
\end{align*}
where we used that $D\geqslant\sum_{n\in\mathbb{N}}a^{n}$, that $v_{m}$ is radial and decreasing for $r\nearrow1$ and $|z|\geqslant r_t$, whence $v_{m}(r_t)^{-1}\leqslant v_{m}(z)^{-1}$. Thus we have
$$
G_{m}(z)\leqslant 2A^2(D+A^2)a_{m}u_{m}(z)^{-1}.
$$
\item[(ii)] We consider $H_{m}(z)$. Then all the occuring $n$ satisfy $n>t+1$ and $n\in N_{m}$. By the estimates in 2.~we obtain
\begin{align*}
H_{m}(z) \stackrel{\text{\tiny dfn}}{=}\sum_{\stackrel{n\in N_{m}}{\scriptscriptstyle n>t+1}}|g_{n}(z)| &= \delta_{k(t+2),m}\,|g_{t+2}| + \sum_{\stackrel{n\in N_{m}}{\scriptscriptstyle n>t+2}}|g_{n}(z)|\\
&\leqslant 4a_{m}v_{m}(z)^{-1}A^2 + \sum_{\stackrel{n\in N_{m}}{\scriptscriptstyle n>t+2}}4\cdot2^{-n}Da_{m}v_{m}(z)^{-1}\\
&\leqslant\big(4A^2+4D\sum_{n=0}^{\infty}2^{-n}\big)a_{m}v_{m}(z)^{-1}\\
&=4(A^2+2D)a_{m}v_{m}(z)^{-1},
\end{align*}
where $\delta$ denotes the Kronecker symbol.
\end{compactitem}
\medskip
Combining the estimates in (i) and (ii) we obtain
\begin{align*}
|h_{m}(z)|&=G_{m}(z)+H_{m}(z)\\
&\leqslant(2A^2(D+A^2)+ 4(A^2+2D))a_{m}v_{m}(z)^{-1}\\
&< C2^{-(m+2)}C^{-1}b_mv_{m}(z)^{-1}\\
&=2^{-(m+2)}b_mv_m(z)^{-1}
\end{align*}
that is $v(z)|h_m(z)|\leqslant2^{-(m+2)}b_m$ and thus
$$
h_m\in 2^{-(m+2)}b_mU_m.
$$
as desired.
\end{proof}

\begin{ack} The author thanks Pepe Bonet for answering many questions, all his useful hints and several clarifying discussions on the proof of \cite[Theorem 5]{BoEnTa2005}.
\end{ack}

\providecommand{\bysame}{\leavevmode\hbox to3em{\hrulefill}\thinspace}
\providecommand{\MR}{\relax\ifhmode\unskip\space\fi MR }
\providecommand{\MRhref}[2]{%
  \href{http://www.ams.org/mathscinet-getitem?mr=#1}{#2}
}
\providecommand{\href}[2]{#2}

\vspace{10pt}
Author's adress: Sven-Ake Wegner, Mathematical Institute, University of Paderborn, D-33095 Paderborn, Germany. E-mail address: \texttt{wegner@math.upb.de}
\end{document}